\documentclass[12pt]{amsart}
\usepackage{amsmath, amssymb, mathabx, enumitem, amsfonts, mathrsfs}
\usepackage{color}
\usepackage[colorlinks=true,linkcolor=blue,urlcolor=blue,citecolor=blue, pagebackref]{hyperref}
\usepackage[alphabetic]{amsrefs}
\usepackage{ulem}
\usepackage{tikz}
\usepackage{tikz-cd}
\usepackage{url}

\usepackage[margin=0.7in]{geometry}

\newcommand{\C}{\mathbb{C}}
\newcommand{\Q}{\mathbb{Q}}
\newcommand{\Z}{\mathbb{Z}}

\newcommand{\op}{\operatorname}

\newcommand{\mf[1]}{\mathfrak{#1}}
\newcommand{\mc[1]}{\mathcal{#1}}

\newcommand{\dom}{\mathcal{P}^+}

\newtheorem{theorem}{Theorem}[section]

\newtheorem{remark}[theorem]{Remark}
\newtheorem{conjecture}[theorem]{Conjecture}

\newtheorem{corollary}[theorem]{Corollary}

\newtheorem{proposition}[theorem]{Proposition}

\newtheorem{lemma}[theorem]{Lemma}
\newtheorem{clm}{Claim}
\newtheorem{definition}[theorem]{Definition}
\newtheorem{definition/lemma}[theorem]{Definition/Lemma}

\DeclareMathOperator{\End}{End}

\title[Components of $V(\rho) \otimes V(\rho)$]{ Components of $V(\rho) \otimes V(\rho)$ and dominant weight polyhedra for affine Kac--Moody Lie algebras}

\author{Sam Jeralds}
\address{School of Mathematics and Physics, University of Queensland, St. Lucia, QLD 4072, AUS}
\email{s.jeralds@uq.edu.au}

\author{Shrawan Kumar}
\address{Department of Mathematics, University of North Carolina, Chapel Hill, NC, 27599-3250, USA}
\email{shrawan@math.unc.edu}

\begin{document}

\begin{abstract} 

 Kostant asked the following question: Let $\mf[g]$ be a simple Lie algebra over the complex numbers. Let $\lambda$ be a dominant integral weight. Then, $V(\lambda)$ is a component of $V(\rho)\otimes V(\rho)$  if and only if $\lambda \leq 2 \rho$  under the usual Bruhat-Chevalley order on the set of weights. In an earlier work with R. Chirivi and A. Maffei the second author gave an  affirmative answer to this question up to a saturation factor. The aim of the current work is to extend this result to untwisted affine Kac-Moody Lie algebra $\mf[g]$ associated to any simple Lie algebra $\mathring{\mf[g]}$ (up to a saturation factor). In fact, we prove the result  for affine $sl_n$ without any saturation factor. Our proof requires some additional techniques including the Goddard-Kent-Olive construction and study of the characteristic cone  of non-compact polyhedra. 

\end{abstract}

\maketitle

\section{Introduction} \label{Intro}
Let $\mf[g]$ be a symmetrizable Kac--Moody Lie algebra over $\C$. In particular, we could let $\mf[g]$ be a (finite-dimensional) semisimple Lie algebra, or an untwisted affine Kac--Moody Lie algebra. Given two  dominant integral  weights $\lambda$ and $\mu$ of $\mf[g]$, the tensor decomposition problem seeks to understand the irreducible components of $V(\lambda) \otimes V(\mu)$, where $V(\lambda)$ is the irreducible representation of $\mf[g]$ with  highest weight $\lambda$, and similarly for $\mu$. This is a classic problem, with approaches and applications in representation theory, geometry, conformal field theory, and algebraic combinatorics, just to name a few. 

In his study of the exterior algebra $\bigwedge \mf[g]$ of  semisimple Lie algebras, Kostant \cite{Kos} found that as $\mf[g]$-representations, $\bigwedge \mf[g]$ decomposes as 
$$
\bigwedge \mf[g] \cong 2^l \left( V(\rho) \otimes V(\rho) \right),
$$
where $l$ is the rank of $\mf[g]$ and $\rho$ is half the sum of positive roots. Because of the role $\bigwedge \mf[g]$ plays in the computation of Lie algebra homology, and also in the structure of the Clifford algebra $\mc[C](\mf[g])$ as was Kostant's motivation, one would like to understand the decomposition of the tensor product $V(\rho) \otimes V(\rho)$ into irreducible components. In this direction, Kostant made the following conjecture, first recorded in the work of Berenstein--Zelevinsky, where they also gave a proof of this conjecture when $\mf[g]=sl_{n+1}$ \cite{BZ}.

\begin{conjecture} \label{Kostant} Let $\mf[g]$ be a semisimple Lie algebra over $\C$. Let $\lambda \in \dom$ be a dominant integral weight such that $\lambda \leq 2\rho$ in dominance order. Then $V(\lambda) \subset V(\rho) \otimes V(\rho)$.
\end{conjecture}

Of course, the opposite implication clearly holds. Later, Chriv\`i--Kumar--Maffei \cite{CKM} proved a weaker version of Kostant's conjecture, which recovers Berenstein--Zelevinsky's result for $sl_{n+1}$, by showing that the expected components appear in the tensor product decomposition ``up to saturation," in the language of the saturated tensor cone; see Section \ref{Tensor} for the precise definitions. 

\begin{theorem}[\cite{CKM}*{Theorem 3}] \label{finite} Let $\mf[g]$ be a (finite-dimensional) semisimple Lie algebra over $\C$. Let $\lambda \in \dom$ be a dominant  integral weight such that $\lambda \leq 2\rho$ in dominance order. Then for any saturation factor $d \geq 1$ of $\mf[g]$, $V(d\lambda) \subset V(d\rho) \otimes V(d\rho)$. In particular, for $\mf[g]=sl_n$, $V(\lambda) \subset V(\rho) \otimes V(\rho)$.
\end{theorem}

The proof of Theorem \ref{finite} in loc. cit. makes use of a set of inequalities due to Berenstein--Sjamaar \cite{BS} which controls the possible irreducible components of $V(\lambda) \otimes V(\mu)$ up to saturation. The tensor decomposition problem is a particular example of a larger class of branching problems, which consider the restriction of representations of $\mf[g]$ to an embedded subalgebra $\mf[g]_1 \hookrightarrow \mf[g]$. In the tensor decomposition case, this corresponds to the diagonal embedding $\mf[g] \hookrightarrow \mf[g] \oplus \mf[g]$. A possible representation-theoretic generalization of Conjecture \ref{Kostant} is to consider the branching of the representation $V(\rho)$ of $\mf[g]$ to a subalgebra $\mf[g]_1$. This was studied by Nadimpalli--Pattanayak \cite{NP}, guided by the approach in \cite{CKM}, generalizing Theorem \ref{finite} via the corresponding inequalities of the related ``branching cone" (see \cite{Ku2} for a survey on the cones and inequalities for the tensor product and more general branching setting). 

We now propose a different representation-theoretic generalization; in particular, we consider the tensor decomposition problem for any untwisted affine Kac--Moody Lie algebra $\mf[g]$. In this setting, one can again try to determine the irreducible components of the tensor product $V(\rho) \otimes V(\rho)$, where $\rho$ now  is the sum of fundamental weights (see Section \ref{Sect2} for conventions). However, as the highest weight irreducible representations in this case are infinite-dimensional, and the tensor product $V(\rho) \otimes V(\rho)$ has infinitely-many components, the problem is more subtle. Nevertheless, we prove the following theorem, which is a verbatim generalization of Theorem \ref{finite} (cf. Theorems \ref{maintheorem} and \ref{maintheorem2} in the text).

\begin{theorem} \label{components} Let $\mf[g]$ be an untwisted affine Kac--Moody Lie algebra, $\lambda \in \dom$ a  dominant  integral
weight such that $\lambda \leq 2\rho$ in dominance order. Then,
\begin{enumerate}
\item for any saturation factor $d \geq 1$ of $\mf[g]$, we have $V(d\lambda) \subset V(d\rho) \otimes V(d\rho)$.
\item if $\mf[g]=A_n^{(1)}=\widehat{sl}_{n+1}$, then $V(\lambda) \subset V(\rho) \otimes V(\rho)$.
\end{enumerate}
\end{theorem}

We note that, unlike in the  semisimple case, Theorem \ref{components}$(2)$ does not follow from \ref{components}$(1)$, as the smallest saturation factor $d$ for $\widehat{sl}_{n+1}$ is $d=2$. While a similar approach via inequalities as in \cite{CKM} and \cite{NP} in the affine Kac--Moody setting would be effective to prove Theorem \ref{components}(1) (see Remark \ref{remark5.3}), such an approach would not address \ref{components}(2).

Instead, because of its utility in the $\mf[g]=\widehat{sl}_{n+1}$ case, our primary tool in proving Theorem \ref{components} is the action of the Virasoro algebra on $V(\rho) \otimes V(\rho)$ via the Goddard--Kent--Olive (or GKO) construction, which we recall in Section \ref{GKO}. The applicability of the GKO construction to the tensor decomposition problem allows us to restrict attention to certain ``maximal" components $V(\nu)$ of $V(\rho) \otimes V(\rho)$, whose existence demonstrates the appearance of the components $V(\nu-k\delta)$ for all $k \geq 0$. This technique originates with Kac--Wakimoto \cite{KW} and has seen applications in the study of the affine tensor semigroup (see Section \ref{semigroup} for references). More recently, the GKO action was crucially used in a similar fashion in our previous work on root components for affine Kac--Moody Lie algebras \cite{JK}.

To apply the GKO construction effectively, we must first understand the dominant weights $\lambda \leq 2\rho$. To do this, we introduce the notion of the \textit{dominant weight polyhedron} $D_\mu$ for a highest weight integrable  (irreducible) representation $V(\mu)$ of an affine Kac--Moody Lie algebra. In the finite case, these are compact polytopes and are completely describable by their vertices. This description was used in \cite{CKM}*{Proposition 9} to express those $\lambda \leq 2\rho$ in a particularly useful form. However, in the affine case, these polyhedra are no longer compact, so cannot be completely determined by their vertices. Taking this into account, we give in Proposition \ref{genvert} an a priori larger set of points that contains the vertices of $D_\mu$ for regular dominant $\mu$; in fact, in Corollary \ref{verticesequal} we show that these points are precisely the vertices when $\mu$ is regular dominant. By a general result for the structure of polyhedra, we then give an explicit decomposition of $D_\mu$ for $\mu$ regular dominant. With this result, we recover a decomposition result for $\lambda \leq 2\rho$ in Proposition \ref{CKMprop} analogous to the one in the semisimple case as in \cite{CKM}.

Finally, we apply the GKO construction to the specific problem of determining the components of $V(\rho) \otimes V(\rho)$ up to  a saturation factor $d$ as in Theorem \ref{components}. We further  propose an exact  analogue of Kostant's conjecture (without any saturation factor)  for affine Kac--Moody Lie algebras in Conjecture \ref{affineconjecture}.

\subsection*{Acknowledgements} The first author thanks M. Besson and J. Kiers for many helpful conversations throughout and Travis Scrimshaw for assistance with computational examples, and acknowledges the support of the Australian Research Council Discovery Project DP200102316. This work was completed while the second author was a member of the Institute for Advanced Study, Princeton, the hospitality of which is gratefully acknowledged.

\section{Notation and conventions for affine Kac--Moody Lie algebras} \label{Sect2}

We fix here briefly the key notation and conventions used throughout the paper. For a full treatment, we refer to \cite{Kac}*{Chapters 6, 7}, whose conventions we adopt.

By $\Z_+$ (resp. $\Q_+$) we mean the nonnegative integers (resp., rational numbers).

Throughout, we denote by $\mf[g]$ the untwisted affine Kac--Moody Lie algebra, associated to a finite-dimensional simple Lie algebra $\mathring{\mf[g]}$. We fix a choice of Cartan subalgebra and Borel subalgebra $\mf[h] \subset \mf[b] \subset \mf[g]$ based on a choice of Cartan subalgebra $\mathring{\mf[h]}$ and a Borel subalgebra $\mathring{\mf[b]}\supset \mathring{\mf[h]}$ of $\mathring{\mf[g]}$. Relative to this choice, we choose a set of simple roots $\Delta=\{\alpha_0, \alpha_1, \dots, \alpha_l\}=\{\alpha_i\}_{i \in I}$, where $l = \op{rank}{\mathring{\mf[g]}}$ and $I:=\{0,1,\dots, l\}$. We set $\Phi$ the set of roots of $\mf[g]$ and denote by $\Phi^+$ (resp., $\Phi^-$) the set of positive roots (resp., negative roots) for this choice. The set of roots $\Phi$ can further be partitioned into the set of real roots $\Phi_{Re}$ and imaginary roots $\Phi_{Im}$, with basic imaginary root $\delta:=\alpha_0+\theta$, where $\theta$ is the highest root of $\mathring{\mf[g]}$. We denote by $d$, $K \in \mf[h]$ the derivation and central elements of $\mf[g]$, respectively. 

The Weyl group of $\mf[g]$ is denoted by $W$, and has simple reflections $\{ s_0, s_1, \dots, s_l\}$ associated to the simple roots. 

We let 
$$\dom :=\{\lambda\in \mf[h]^*: \langle \lambda, \alpha_i^\vee\rangle \in \Z_+ \,\forall i\in I\}$$
 be the set of  dominant integral weights of $\mf[g]$. The fundamental weights $\Lambda_i \in \dom$ are defined uniquely by $\langle \Lambda_i, \alpha_j^\vee\rangle=\delta_{ij}$, the Kronecker delta, for all simple coroots $\alpha_j^\vee$ and $\langle \Lambda_i, d\rangle=0$. Then the corresponding Weyl vector $\rho \in \dom$ is given by $\rho=\sum_{i=0}^l \Lambda_i$ (so that $\langle \rho, d\rangle=0$). For any $\lambda \in \dom$, the value of $\langle \lambda, K\rangle$ is called the level of $\lambda$; throughout, we will restrict to those $\lambda \in \dom$ with \textbf{positive} level. 

The weight lattice $\mc[P]$ and root lattice $\mc[Q]$ are given by 
$$
\mc[P]:=\bigoplus_{i=0}^l \Z \Lambda_i \oplus \C \delta,\,\,\,
\mc[Q]:= \bigoplus_{i=0}^l \Z \alpha_i.
$$

Finally, the similarly defined objects for $\mathring{\mf[g]}$ will be denoted by a circle; i.e., the root lattice $\mathring{\mc[Q]}$ of $\mathring{\mf[g]}$, etc. We fix also a normalized invariant form $( \cdot | \cdot)$ on $\mf[g]$ such that $(\theta | \theta) =2$, where again $\theta \in \mathring{\Phi}^+$ is the highest root of the underlying semisimple Lie algebra $\mathring{\mf[g]}$; we will always use this choice of normalization.

\section{Irreducible highest weight representations and dominant weight polyhedra} \label{Sect3}

In this section, we recall some basic results on the structure of irreducible highest weight representations of $\mf[g]$. While many of the results mentioned hold more generally at the level of symmetrizable Kac--Moody algebras, we focus only on statements for the affine Lie algebra setting. We use \cite{Kac}*{Chapters 11, 12} as our primary sources.

\subsection{Weights of irreducible highest weight modules} 

Let $\lambda \in \dom$ be a dominant integral weight of $\mf[g]$. We denote by $V(\lambda)$ the corresponding irreducible, highest weight representation of $\mf[g]$ with highest weight $\lambda$. As an $\mf[h]$-module via restriction, $V(\lambda)$ has a decomposition into weight spaces 
$$
V(\lambda)=\bigoplus_{\mu \in \mc[P]} V(\lambda)_\mu,
$$
where $V(\lambda)_\mu:=\{v \in V(\lambda): h.v=\langle \mu, h\rangle v \ \forall h \in \mf[h]\}$. We set $\mc[P](\lambda)$ to be the set of weights of the representation $V(\lambda)$; that is, $\mc[P](\lambda)=\{\mu \in \mc[P]: V(\lambda)_\mu \neq 0\}$. This is a (typically infinite) $W$-invariant set of weights in $\mc[P]$.

For any two weights $\lambda, \mu \in \mc[P]$, we say that $\lambda \geq \mu$ if $\lambda-\mu \in \mc[Q]^+:= \bigoplus_{i=0}^l \Z_+ \alpha_i$. This defines a partial order on the weight lattice $\mc[P]$, the \textit{dominance order}. As $V(\lambda)$ is a highest weight representation, for any $\mu \in \mc[P](\lambda)$, necessarily $\mu \leq \lambda$. While it is not immediately clear from an algebraic perspective for which $\mu \leq \lambda$ we have $\mu \in \mc[P](\lambda)$, this question has a nice combinatorial answer. We record now the following standard proposition, taken from \cite{Kac}*{Proposition 12.5(a)(b)}, which gives a concise description of $\mc[P](\lambda)$.

\begin{proposition} \label{convhull}
For $\lambda\in \dom$, let $V(\lambda)$ be the irreducible highest weight representation with highest weight $\lambda$, and set of weights $\mc[P](\lambda)$. Then 
$$
\mc[P](\lambda)=W \cdot \{ \mu \in \dom: \mu \leq \lambda\} = (\lambda+\mc[Q]) \cap \op{conv}_\Q \{w\lambda: w \in W\},$$
where $\op{conv}_\Q$ denotes the rational convex hull.
\end{proposition}

\subsection{$\delta$-maximal weights}
Let $\lambda \in \dom$ be a dominant  integral  weight, and consider the associated highest weight irreducible representation $V(\lambda)$ with set of weights $\mc[P](\lambda)$. We recall the following definition.

\begin{definition} A weight $\mu \in \mc[P](\lambda)$ is called $\delta$-maximal if $\mu +k \delta \not \in \mc[P](\lambda)$ for any $k >0$. We denote by $\mc[P]_{max}(\lambda)$ the set of all $\delta$-maximal weights. 
\end{definition}

Since $\delta$ is fixed by all Weyl group elements $w \in W$, the set of $\delta$-maximal weights in $\mc[P](\lambda)$ is $W$-invariant. But as any weight in $\mc[P](\lambda)$ can be $W$-translated to a  dominant weight (cf. \cite{Ku3}*{Proposition 1.4.2(c)} together with Proposition \ref{convhull}), it suffices to understand the set of $\delta$-maximal \textit{dominant} weights; we denote this set by $\dom_{max}(\lambda)$. The study of $\delta$-maximal dominant weights plays a crucial role in applications of representation theory to modular forms and conformal field theory, via understanding (up to renormalization) the ``string functions" 
$$
t \mapsto \op{dim}(V(\lambda)_{\mu-t\delta})
$$
for $\mu \in \dom_{max}(\lambda)$ (see \cite{Kac}*{Chapter 12}). As these will play an important role in what is to come, we record the following lemma on the $``\delta$-strings" $\{\mu-k\delta\}_{k\geq 0} \subset \mc[P](\lambda)$, which is a reformulation of \cite{Kac}*{Prop. 12.5(e)}.

\begin{lemma} \label{unbroken} 
Let $\lambda \in \dom$ and $\mu \in \mc[P](\lambda)$. Then $\mu-k\delta \in \mc[P](\lambda)$ for all $k \geq 0$. 
\end{lemma}

Combining Lemma \ref{unbroken} and Proposition \ref{convhull}, we can get the following description of $\delta$-maximal dominant weights, taken from \cite{BK}*{Proposition 4.4}.

\begin{proposition} \label{maxsupport}
Let $a_0, a_1, \dots, a_l$ be defined by $\delta=\sum_{i \in I} a_i \alpha_i$ and let $\lambda \in \dom$. Then $\mc[P]^+_{max}(\lambda)$ is the set of all $\mu \in \dom$ such that 
\begin{enumerate}
\item $\mu \leq \lambda$, and 
\item if $\lambda-\mu= \sum_{i \in I} c_i \alpha_i$ then there exists some $i \in I$ such that $c_i < a_i$. 
\end{enumerate}
\end{proposition}

\begin{proof}
Suppose for contradiction that $\mu \in \mc[P]^+_{max}(\lambda)$ and $\lambda - \mu = \sum_{i \in I} c_i \alpha_i$ with $c_i \geq a_i$ for all $i$. Set $\mu':=\mu+\delta \in \dom$. Then we have 
$$
\lambda-\mu' = \lambda-\mu-\delta=\sum_{i \in I} (c_i-a_i) \alpha_i.
$$
Note that since $c_i - a_i \geq 0$ for all $i$ by assumption, we get that $\lambda - \mu' \in \mc[Q]^+$, or that $\mu' \leq \lambda$ in dominance order. But then by Proposition \ref{convhull}, we get that $\mu' \in \mc[P](\lambda)$, which contradicts $\mu$ being $\delta$-maximal. The converse is clear. 
\end{proof}

\subsection{Dominant weight polyhedra}
We next want to better understand the dominant weights $\mu \in \dom$ such that $\mu \in \mc[P](\lambda)$. By Lemma \ref{unbroken}, it would suffice to understand those weights in $\dom_{max}(\lambda)$. However, this is in general a subtle issue, as it is not straightforward to describe uniformly the set $\dom_{max}(\lambda)$ in terms of $\lambda$. We instead consider a larger set, which we refer to as the \textit{dominant weight polyhedron} associated to $\lambda$, and denote it  by $D_\lambda$. This will be a rational, (generically) unbounded polyhedron in the sense of \cite{Sch}, which we take as our standard reference. 

We restrict for our purposes to the case when $\lambda \in \dom$ is regular dominant, so that $\langle \lambda, \alpha_i^\vee \rangle>0$ for all $i \in I:=\{0, 1, \dots, l\}$. We will also assume that $\langle \lambda, d \rangle=0$, up to replacing $\lambda$ with $\lambda-\langle \lambda, d\rangle \delta$.  Let $\mc[C]_\Q$ denote the rational dominant chamber, defined by 
$$
\mc[C]_\Q:=\{\mu \in \mf[h]^\ast_\Q: \ \langle \mu, \alpha_i^\vee \rangle \geq 0 \ \forall i \in I\},
$$
where $ \mf[h]^*_{\Z} :=\left( \oplus_{i=0}^l \Z\Lambda_i\right) \oplus \Z \delta$ and $\mf[h]^*_\Q:= \Q\otimes_\Z\mf[h]^*_\Z$. 
We now fix the following definition. 

\begin{definition} For a dominant integral weight $\lambda \in \dom$, the dominant weight polyhedron $D_\lambda$ is defined by 
$$
D_\lambda:= \mc[C]_\Q \cap \op{conv}_\Q \{w \lambda: w \in W\}.
$$
\end{definition}

Thus, by Proposition \ref{convhull}, the points $\mu \in D_\lambda$ correspond to rational dominant weights such that $N\mu \in \mc[P](N \lambda) \cap \dom$ for some scaling factor $N \in \Z_{>0}$. Again by Proposition \ref{convhull}, the points $\mu \in D_\lambda$ are determined by the set of inequalities for $\mu \in \mf[h]^*_\Q$:
$$
\begin{cases}
\langle \mu, \alpha_i^\vee \rangle \geq 0 \ \forall i \in I, \\
\langle \lambda-\mu, x_i \rangle \geq 0 \ \forall i \in I, \\
\langle \lambda-\mu, K\rangle =0,\\
\end{cases}
$$
where $x_i$ is the fundamental coweight defined by $\langle \alpha_j, x_i \rangle=\delta_{i,j}$. The faces of $D_\lambda$ are given by replacing a subset of the above inequalities by equalities. To this end, let $J \subset I$ and and consider the sets
$$
\begin{aligned}
A_J &:= \sum_{k \in J} \Q_+ \Lambda_k, \\
\overline{A}_J &:= A_J + \Q_+(-\delta), \\
B_J &:= \sum_{k \in J} \Q_+ \alpha_k, \\
C_J= C_J(\lambda) &:= \lambda-B_J. 
\end{aligned}
$$
In particular we set  $\overline{A}=\overline{A}_I= \sum_{k \in I} \Q_+ \Lambda_k + \Q_+(-\delta)$, $B= B_I= \sum_{k \in I} \Q_+ \alpha_k$, and $C=C_I=\lambda-B$. Then, $\overline{A}_H\cap C_K$ are faces of $D_\lambda$  for $H, K\subset I$.  By construction, we have that 
\begin{equation}\label{eqn3.3.1}
D_\lambda=\overline{A} \cap C;
\end{equation}
note that as any $\mu \in D_\lambda$ satisfies $\langle \mu, d \rangle \leq \langle \lambda, d\rangle$ since  $N\mu \leq N \lambda$ (for some positive integer $N$)  in dominance order, the restriction to $\Q_+(-\delta)$ in the definition of $\overline{A}$ suffices. 

Unlike when considering a finite-dimensional semisimple Lie algebra $\mathring{\mf[g]}$, the polyhedra $D_\lambda$ are typically unbounded. First, we want to determine the ``infinite directions" of the polyhedron $D_\lambda$. We introduce the following definition from \cite{Sch}*{Section 8.2}.

\begin{definition}
For a polyhedron $P$, the characteristic cone $\op{char.cone}(P)$ is given by 
$$
\op{char.cone}(P):= \{y : x+y \in P \ \forall x \in P \}.
$$
\end{definition}

\begin{lemma} \label{charcone}
For $\lambda \in \dom$, we have $\op{char.cone}(D_\lambda)=\Q_+(-\delta)$. 
\end{lemma}

\begin{proof}
As in (\ref{eqn3.3.1}),  $D_\lambda = \overline{A} \cap C$. Clearly $\Q_+ (-\delta) \subseteq \op{char.cone}(\overline{A} \cap C)$. Let $y \in \mf[h]^\ast_\Q$ such that $y \in \op{char.cone}(\overline{A} \cap C)$. Since $\lambda \in\overline{A} \cap C$, we immediately  conclude $y \in -\sum_{k \geq 0} \Q_+ \alpha_k$, by the definition of $C$.

Assume that $\langle y, \alpha_j^\vee \rangle <0$ for some $j$. Then for any $x \in \overline{A} \cap C$, we have $x+Ny \in \overline{A} \cap C$ for any $N \in \Z_+$. Taking $N$ sufficiently large, we get that 
$$
\langle x+Ny, \alpha_j^\vee \rangle = \langle x, \alpha_j^\vee \rangle +N \langle y, \alpha_j^\vee \rangle <0.
$$
Contradiction, since $x+Ny \in \overline{A}$, and any $z \in \overline{A}$ satisfies $\langle z, \alpha_j^\vee \rangle \geq 0$ for all $j$. Thus, $\langle y, \alpha_j^\vee \rangle \geq 0$ for all $j$.

Next, suppose that $\langle y, \alpha_j^\vee \rangle >0 $ for some $j$. Note that we can write the canonical central element $K$ of $\mf[g]$ as $K= \sum_{k \geq 0} a^\vee_k \alpha_k^\vee$ for the appropriate positive labels $a_k^\vee$,  and as $y$ pairs nonnegatively with all $\alpha_k^\vee$, we get 
$$
\langle y, K \rangle >0;
$$
contradiction, since $y \in -\sum_{k \geq 0} \Q_+ \alpha_k$, and $\alpha_k$ pairs to zero with $K$ for all $k \geq 0$. 

Thus, $\langle y, \alpha_k^\vee \rangle =0$ for all $k \geq 0$. Now, using that $\mf[h]^\ast_\Q = \mathring{\mf[h]}^\ast_\Q \oplus \left( \Q (\delta) \oplus \Q (\Lambda_0) \right)$, we get that 
$$
y \in \mathring{\mf[h]}^\ast_\Q \oplus \Q (\delta) 
$$
and is orthogonal to all of $\mathring{\mf[h]}^\ast_\Q$; thus necessarily $y \in \Q(\delta)$. In particular, this forces $y \in \Q_+(-\delta)$, as desired. 
\end{proof}

We are next interested in the minimal faces of $D_\lambda$. By the general theory of polyhedra \cite{Sch}*{Section 8.5(22)}, the minimal faces of a polyhedron $P$ are translates of the ``lineality space" 
$$
\op{lin. space}(P):=\op{char. cone}(P) \cap - \op{char. cone}(P).
$$
Then by Lemma \ref{charcone}, in the case of $D_\lambda$, we get that $\op{lin. space}(D_\lambda)=\{0\}$. Thus the minimal faces of $D_\lambda$ are zero-dimensional, hence vertices. To understand these vertices, we will examine the various face intersections $A_H \cap C_K$ and $\overline{A}_H \cap C_K$, for subsets $H, K \subseteq I$, and analyze those which intersect in single points. First, in the next lemma we introduce the possible  candidates for the vertices of $D_\lambda$. 

\begin{lemma} \label{vertexdefinition}
Let $\lambda \in \dom$ be a regular dominant integral weight. For any subset $J \subsetneq I$, let $W_J=\langle s_j : j \in J \rangle \subset W$ be the corresponding (finite) parabolic subgroup corresponding to the simple roots $\{\alpha_j\}_{j \in J}$. Define
$$
\hat{b}_J(\lambda):= \frac{1}{|W_J|} \sum_{w \in W_J} w(\lambda).
$$
Then $\hat{b}_J(\lambda) \in D_\lambda$. In fact, for each $J \subsetneq I$ we have $\hat{b}_J(\lambda) \in \overline{A}_{I \backslash J}$, and $\hat{b}_J(\lambda) \neq \hat{b}_{J'}(\lambda)$ for $J \neq J'$.
\end{lemma}

\begin{proof}
First, we show that $\hat{b}_J(\lambda) \in D_\lambda$. To this end, by its definition, $\hat{b}_J (\lambda)\in  \op{conv}_\Q \{w \lambda: w \in W\}$.
It remains to show that $\hat{b}_J(\lambda) \in \mc[C]_\Q$.  Let $\alpha_i^\vee$ be a simple coroot. We consider two cases: first, suppose that $i \in J$. Then we have that 
\begin{equation} \label{vertvanishing}
|W_J| \langle \hat{b}_J(\lambda), \alpha_i^\vee \rangle = \left \langle \sum_{w \in W_J} w(\lambda), \alpha_i^\vee \right \rangle = \left \langle \lambda, \sum_{w \in W_J} w^{-1}(\alpha_i^\vee) \right \rangle=0,
\end{equation}
since $i \in J$ so that $\sum_{w \in W_J} w^{-1}(\alpha_i^\vee)=0$. Else if $i \not \in J$, then for all $w \in W_J$, we have $w^{-1}(\alpha_i^\vee)$ is a positive coroot, hence
\begin{equation} \label{vertnonzero}
|W_J| \langle \hat{b}_J(\lambda), \alpha_i^\vee \rangle = \left \langle \sum_{w \in W_J} w(\lambda), \alpha_i^\vee \right \rangle = \left \langle \lambda, \sum_{w \in W_J} w^{-1}(\alpha_i^\vee) \right \rangle >0,
\end{equation}
since $\lambda$ is regular dominant. Thus, $\hat{b}_J(\lambda) \in \mc[C]_\Q$, as desired, and thus $\hat{b}_J(\lambda) \in D_\lambda$.

 Further, since by (\ref{vertvanishing}) and (\ref{vertnonzero}),
 $\hat{b}_J(\lambda)$ is a rational dominant weight that vanishes precisely on $\{\alpha_j^\vee\}_{j \in J}$, we get that $\hat{b}_J(\lambda) \in \overline{A}_{I \backslash J}$ by definition,  since $\langle w\lambda, d\rangle \leq \langle \lambda, d\rangle =0$.

 Finally, if $j \in J \backslash J'$ we have $\langle \hat{b}_J(\lambda), \alpha_j^\vee \rangle =0$ while $\langle \hat{b}_{J'}(\lambda), \alpha_j^\vee \rangle >0$, by (\ref{vertvanishing}) and (\ref{vertnonzero}), so $\hat{b}_J(\lambda) \neq \hat{b}_{J'}(\lambda)$.
\end{proof}

\begin{proposition} \label{genvert}
Let $\lambda \in \dom$ be a regular dominant integral weight, and assume that $\langle \lambda, d \rangle =0$. Then the vertices of $D_\lambda$ are contained in the set $\{ \hat{b}_J(\lambda)\}_{J \subsetneq I}$. 
\end{proposition}

\begin{proof}
As shown above,    the minimal faces of $D_\lambda$ are points, hence vertices.  We show that they are contained in the set $\{\hat{b}_J(\lambda)\}_{J \subsetneq I}$. This is done by considering the following ten claims.

\begin{clm} \label{C1} If $H \not \subseteq K$, then $\overline{A}_{I\backslash H} \cap C_K = \varnothing$.
\end{clm}

Take $y \in \overline{A}_{I \backslash H} \cap C_K$ and let $y =\lambda -x$ for some $x \in B_K$. Write 
\begin{equation} \label{claim 1}
x=\lambda-y=\lambda-\sum_{h \not \in H} a_h \Lambda_h +b \delta 
\end{equation}
where $x \in \sum_{k \in K} \Q_+ \alpha_k$,  $a_h, b \in \Q_+$. Take $j \in H \backslash K$. Evaluating (\ref{claim 1}) at $\alpha_j^\vee$, we get a contradiction.

\begin{clm} \label{C2} If $H \subsetneq K$ and $0 \not \in K$, then $A_{I \backslash H} \cap C_K$ contains $\hat{b}_H(\lambda)$ and $\hat{b}_K(\lambda)$.
\end{clm}

Clearly $\hat{b}_H(\lambda)$, $\hat{b}_K(\lambda) \in C_K$. By Lemma \ref{vertexdefinition}, we know that $\hat{b}_H(\lambda)$, $\hat{b}_K(\lambda) \in \overline{A}_{I \backslash H}$. Now, since $0 \not \in K$ (and thus also not in $H$), we get that $\langle \hat{b}_K(\lambda), d \rangle = \langle \hat{b}_H(\lambda), d \rangle =0$, so that $\hat{b}_K(\lambda), \hat{b}_H(\lambda) \in A_{I \backslash H}$.

\begin{clm} \label{C3} For any subsets $H, K \subset I$, $A_{I \backslash H} \cap C_K = A_{I \backslash H} \cap C_{K \backslash \{0\} }$.
\end{clm}

This follows easily since for any $x \in A_{I \backslash H}$, $\langle x, d \rangle=0$, whereas $\langle \alpha_0, d \rangle=1$.

\begin{clm} \label{C4} For $H \subsetneq K \subsetneq I$ and $0 \in H$, $\hat{b}_H(\lambda)$, $\hat{b}_K(\lambda) \in \overline{A}_{I \backslash H} \cap C_K$ and $A_{I \backslash H} \cap C_K = \varnothing$.
\end{clm}

Of course, $\hat{b}_H(\lambda)$, $\hat{b}_K(\lambda) \in C_K$, and by Lemma \ref{vertexdefinition}, $\hat{b}_H(\lambda)$, $\hat{b}_K(\lambda) \in \overline{A}_{I \backslash H}$. For the second part, by Claim \ref{C3}, $A_{I \backslash H} \cap C_K = A_{I \backslash H} \cap C_{K \backslash \{0\} } = \varnothing$, by Claim \ref{C1} since $0 \in H$.

\begin{clm} \label{C5} For $H \subsetneq K \subsetneq I$ and $0 \in K \backslash H$, $\hat{b}_H(\lambda)$, $\hat{b}_{K \backslash \{0\} }(\lambda) \in A_{I \backslash H} \cap C_K$, even if $H=K \backslash \{0\}$. 
\end{clm}

By Claim \ref{C3}, $A_{I \backslash H} \cap C_K = A_{I \backslash H} \cap C_{K \backslash \{0\}}$. Further $H \subset K \backslash \{0\}$ since $0 \not \in H$. Thus, the claim follows from Claim \ref{C2}. (Since $0\notin H$, it is easy to see that  $\hat{b}_H(\lambda) \in A_{I \backslash H}$.)

\begin{clm} \label{C6} For $H \subsetneq K =I$ and $H \subset I \backslash \{i_1, i_2\}$,  none of $A_{I \backslash H} \cap C_I$ or $\overline{A}_{I \backslash H} \cap C_I$ contain exactly one point. 
\end{clm}

If $i_1=0$, $A_{I \backslash H} \cap C_I = A_{I \backslash H} \cap C_{I \backslash \{0\} }$ contains at least two points by Claim \ref{C2} and hence so does $\overline{A}_{I \backslash H} \cap C_I$. If neither of $i_1, i_2$ is $0$ (i.e., we can assume that $0 \in H$), then 
$$
A_{I \backslash H} \cap C_I = A_{I \backslash H} \cap C_{I \backslash \{0\}} = \varnothing,
$$
where the equalities follow from Claims \ref{C3} and \ref{C1} respectively. Further,
$$
\overline{A}_{I \backslash H} \cap C_I \supset \overline{A}_{I \backslash H} \cap C_{I \backslash \{i_1\} } \supset \{\hat{b}_H(\lambda), \hat{b}_{I \backslash \{i_1\}}(\lambda) \}$$
by Claim \ref{C4}.

\begin{clm} \label{C7} For $H=I \backslash \{i_0\}$, $K=I$: if $i_0=0$, then $A_{I \backslash H} \cap C_I$ contains the point $\hat{b}_{I \backslash \{0\}}(\lambda)$. If $i_0 \neq 0$, then $A_{I \backslash H} \cap C_I = \varnothing$ and $\overline{A}_{I \backslash H} \cap C_I$ contains the point $\hat{b}_{I \backslash \{i_0\}}(\lambda)$.
\end{clm}

Let $i_0=0$. By Lemma \ref{vertexdefinition}, we have $\hat{b}_{I \backslash \{0\}}(\lambda) \in \overline{A}_{\{0\}}$. But since $H=I \backslash \{0\}$, necessarily $\langle \hat{b}_{I \backslash \{0\}}(\lambda), d \rangle = \langle \lambda, d \rangle =0 $, so in fact $\hat{b}_{I\backslash \{0\}}(\lambda) \in A_{\{0\}}$. Of course, $\hat{b}_{I \backslash \{0\}}(\lambda) \in C_I$. 

For $i_0 \neq 0$, by Claim \ref{C3}, $A_{I \backslash H} \cap C_I = A_{I \backslash H} \cap C_{I \backslash \{0\}}$, and since $0 \in H$, by Claim \ref{C1} this is empty. Now, again by Lemma \ref{vertexdefinition}, $\hat{b}_{I \backslash \{i_0\}}(\lambda) \in \overline{A}_{I \backslash H} \cap C_I$ for $i_0 \neq 0$.

\begin{clm} \label{C8} For any $H$ such that $0 \not \in H$, $A_{I \backslash H} \cap C_H = \overline{A}_{I \backslash H} \cap C_H$ contains the point $\hat{b}_{H}(\lambda)$.
\end{clm}

The identity $A_{I \backslash H} \cap C_H = \overline{A}_{I \backslash H} \cap C_H$ follows since, for any $y \in C_H$, $\langle y, d \rangle=0$. Further, $\hat{b}_H(\lambda) \in C_H$ by construction and $\hat{b}_H(\lambda) \in \overline{A}_{I \backslash H}$ by Lemma \ref{vertexdefinition}.

\begin{clm} \label{C9} For any $H$ such that $0\in H \subsetneq I$, $A_{I \backslash H} \cap C_H = \varnothing$ and $\overline{A}_{I \backslash H} \cap C_H$ contains the point $\hat{b}_H(\lambda)$. 
\end{clm}

$A_{I \backslash H} \cap C_H = A_{I \backslash H} \cap C_{H \backslash \{0\}}$, by Claim \ref{C3}. But by Claim \ref{C1}, this is empty. Finally, $\hat{b}_H(\lambda) \in C_H$ by construction, and by Lemma \ref{vertexdefinition}, $\hat{b}_H(\lambda) \in \overline{A}_{I \backslash H}$.

\begin{clm} \label{C10} $\overline{A}_{I \backslash I} \cap C_I = \varnothing$.
\end{clm}

We need to find $\{ a_i\}_{i \in I}$, $a_i \geq 0$, such that 
$$
\lambda- \sum_{i \in I} a_i \alpha_i \in \Q_+(-\delta).
$$
Now, the above implies $\langle \lambda, \alpha_j^\vee \rangle=\sum_{i \in I} a_i \langle \alpha_i, \alpha_j^\vee \rangle$ for all $\alpha_j^\vee$. But this can be rewritten as 
\begin{equation} \label{claim 10}
\Lambda=A \cdot C,
\end{equation}
where $\Lambda=(\lambda_0, \lambda_1,\lambda_2,\dots,\lambda_l)$, for $\lambda_j := \langle \lambda, \alpha_j^\vee \rangle >0$, $A=(a_0, a_1, \dots, a_l)$, and $C=\left( \alpha_i(\alpha_j^\vee) \right)_{i,j\geq 0}$ is the affine Cartan matrix. But (\ref{claim 10}) is contradicted by \cite{Kac}*{Theorem 4.3}. \\

Combining all the cases covered by Claims 1--10, we get that if any of $A_{I \backslash H} \cap C_K$ or $\overline{A}_{I \backslash H} \cap C_K$ for any $H, K \subseteq I$ is a single point, then that point is one of the points from the set $\{ \hat{b}_H(\lambda)\}_{H \subsetneq I}$. 
\end{proof}

\begin{corollary}  \label{verticesequal} Following the notation  and assumptions as in Proposition \ref{genvert}, 
the set of vertices of $D_\lambda$ is equal to (not just contained in) $\{\hat{b}_H(\lambda)\}_{H \subsetneq I}$.
\end{corollary}

\begin{proof}  Take any $H\subsetneq I$ and  let $\mu \in  \overline{A}_{I \backslash H} \cap C_H $. Then, we can write
$$\mu=\sum_{i\notin H}m_i\Lambda_i -m\delta=\lambda -\sum_{k\in H}n_k\alpha_k,\,\,\,\text{for some  $m_i, m, n_k\in \Q_+$}.$$
Evaluating the above at $\alpha_i^\vee$, for $i\in H$, we get:
\begin{equation} \label{eqn3.10.1}
\langle \lambda, \alpha_i^\vee\rangle -\sum_{k\in H} n_k \langle \alpha_k, \alpha_i^\vee\rangle=0, \,\,\,\text{for any $i\in H$}.
\end{equation}
Consider the row matrices
$\vec{\lambda}_H := \left(\langle \lambda, \alpha_i^\vee\rangle\right)_{i\in H}; \,\, \vec{n}:= \left(n_k\right)_{k\in H}$
and the $|H|\times |H|$-matrix $C_H:= \left(\langle \alpha_k, \alpha_i^\vee\rangle\right)_{k,i\in H}$. Since $H$ is a proper subset of $I$, $C_H$ is the Cartan matrix of a semisimple Lie algebra. In particular, it is invertible. Hence the above equation (\ref{eqn3.10.1}) gives
$$\vec{\lambda}_H = \vec{n}  \cdot C_H,\,\,\text{i.e.,}\,\,       \vec{\lambda}_H\cdot C_H^{-1}= \vec{n}.$$
Thus, $\vec{n}$ is unique if it exists satisfying $n_k\in \Q_+$ for all $k\in H$. This shows that  $ \overline{A}_{I \backslash H} \cap C_H $ has at most one point. However, as shown in the Claims \ref{C8} and \ref{C9} of the proof of Proposition \ref{genvert}, it has at least the point 
 $\hat{b}_H(\lambda)$. This proves the lemma.
\end{proof}

Specializing to $\lambda=2\rho$, we can simplify the above to get the following corollary. 

\begin{corollary} \label{2rhovert}
The vertices of $D_{2\rho}$ are given by the set $\{\hat{b}_J(2\rho)\}_{J \subsetneq I}$, where
$$
\hat{b}_J(2\rho)=\rho+w_0^J(\rho),
$$
with $w_0^J \in W_J$ the longest element.
\end{corollary}

\begin{proof}
By definition, we have 
$$
\hat{b}_J(2\rho)= \frac{1}{|W_J|} \sum_{w \in W_J} w(2\rho).
$$
We rewrite this as 
$$
\hat{b}_J(2\rho) = 2\rho+ \frac{1}{|W_J|} \sum_{w \in W_J} \left( w(2\rho) -2\rho \right) = 2\rho + \frac{2}{|W_J|} \sum_{w \in W_J} w(\rho)-\rho.
$$
Now, we have 
$$
w(\rho)-\rho = -\sum_{\beta \in \Phi_J(w)} \beta,
$$
where $\Phi_J(w):=\{\beta \in \Phi^+_J: w^{-1}(\beta) \in -\Phi^+_J\}$ and $\Phi^+_J$ are the positive roots in the sub-root system corresponding to $J$. Every root $\beta \in \Phi^+_J$ appears in precisely half of the sets $\Phi_J(w)$ as $w$ varies across $W_J$ (since for any $\beta\in \Phi^+_J$ and $w\in W_J$, $\beta$ belongs to exactly one of $\Phi_J(w)$ or $\Phi_J(ww^J_0)$); thus we get 
$$
\hat{b}_J(2\rho) = 2\rho -\frac{2}{|W_J|} \sum_{\beta \in \Phi^+_J} \frac{|W_J|}{2} \beta = 2\rho-\sum_{\beta \in \Phi^+_J} \beta = 2\rho+(w_0^J(\rho)-\rho) = \rho+w_0^J(\rho).
$$ 
\end{proof}

Finally, by the decomposition theorem for polyhedra in terms of minimal faces and the characteristic cone (cf. \cite{Sch}*{Section 8.9}), we get the following crucial proposition. 

\begin{proposition} \label{dompoly} For any $\lambda \in \dom$ regular dominant, the dominant weight polyhedron $D_\lambda$ decomposes as 
$$
D_\lambda = \op{conv}_\Q \left(\{\hat{b}_J(\lambda)\}_{J \subsetneq I}\right) + \Q_+(-\delta).
$$
\end{proposition}

\begin{remark} \label{coneapproach} In the case of a finite-dimensional semisimple Lie algebra, the vertices of the corresponding dominant weight polytope for $2\rho$ were worked out by Kostant, as reported in \cite{BZ}. For any dominant weight, formulas for the vertices resembling those in Proposition \ref{genvert} were given in \cite{BJK}. In this case, the vertices completely determine the polytope, as it is compact. While we will not need it here, we conjecture the following stronger results.
\end{remark}

\begin{conjecture} 
Let $\lambda \in \dom$ be an arbitrary (not necessarily regular) dominant integral weight with positive level. Then the vertices of $D_\lambda$ are given precisely by the set $\{\hat{b}_J(\lambda)\}_{J \subsetneq I}$ (which may have repeated elements), and $D_\lambda$ decomposes as 
$$
D_\lambda = \op{conv}_\Q \left(\{\hat{b}_J(\lambda) \}_{J \subsetneq I} \right) + \Q_+ (-\delta).
$$
\end{conjecture}

\section{Tensor product of irreducible representations and the Goddard--Kent--Olive construction} \label{Tensor}

In this section, we will review some of the key features of the tensor decomposition problem, which seeks to understand the irreducible components of $V(\lambda)\otimes V(\mu)$ for $\lambda$, $\mu \in \dom$. We will also briefly introduce an action of the Virasoro algebra on $V(\lambda) \otimes V(\mu)$, given by the \textit{Goddard--Kent--Olive} (or GKO) construction, and its application to the tensor decomposition problem. 

Throughout, we will take $\lambda$, $\mu \in \dom$ with $\lambda(d)=\mu(d)=0$, without loss of generality, as we could replace $\lambda$ (and similarly $\mu$) up to twisting by $k\delta$ via $V(\lambda) \otimes V(k\delta)=V(\lambda+k\delta)$.

\subsection{The tensor semigroup} \label{semigroup}
We return now to the tensor decomposition problem, as discussed in the introduction. To this end, consider $\lambda, \mu \in \dom$ and the tensor product decomposition 
$$
V(\lambda) \otimes V(\mu) \cong \bigoplus_{\nu \in \dom} V(\nu)^{m_{\lambda, \mu}^\nu},
$$
where $m_{\lambda, \mu}^\nu \in \Z_+$ is the multiplicity of the subrepresentation $V(\nu)$ in $V(\lambda) \otimes V(\mu)$. While the problem of determining for which $\nu$ we have $m_{\lambda, \mu}^\nu \neq 0$ (or in fact determining $m_{\lambda, \mu}^\nu$ precisely) has a long history which has developed exact methods to give exact solutions (crystals, the Littelmann path model, among others), predicting components which appear in $V(\lambda)\otimes V(\mu)$ based only off of the pair $(\lambda, \mu)$ is challenging. One such example of components, known as the Parthasarathy--Ranga Rao--Varadarajan or PRV components, is given in the following theorem originally due to Kumar \cite{Ku0} for semisimple Lie algebras. Its extension to any symmetrizable Kac-Moody Lie algebra was proved in \cite{Ku1} and Mathieu \cite{Mat}. 

\begin{theorem} \label{PRV} 
Let $\lambda$, $\mu \in \dom$ be two dominant integral weights and $w$, $v \in W$ two Weyl group elements such that $\nu:=w\lambda+v\mu \in \dom$. Then $V(\nu) \subset V(\lambda) \otimes V(\mu)$.
\end{theorem}

As above, we are interested in the set of triples $\{(\lambda, \mu; \nu) \in (\dom)^3$: $m_{\lambda, \mu}^\nu \neq 0$\}. This set in fact forms a semigroup, which we refer to as the tensor semigroup for $\mf[g]$. This follows, for example, from the analogue of the Borel--Weil theorem for symmetrizable  Kac--Moody Lie algebras (\cite{Ku3}*{Corollary 8.3.12}). The tensor semigroup and its generators are difficult to understand, even in the finite-dimensional semisimple case. We consider a simpler  object, known as the saturated tensor semigroup or saturated tensor cone, in the following definition. 

\begin{definition} \label{cone}
The saturated tensor cone associated to $\mf[g]$, denoted $\Gamma(\mf[g])$, is given by 
$$
\Gamma(\mf[g]):=\{(\lambda, \mu; \nu) \in (\dom_\Q)^3: \exists N \geq 1 \text{with } m_{N\lambda, N\mu}^{N\nu} \neq 0 \};
$$
that is, $V(N\nu) \subset V(N \lambda) \otimes V(N\mu)$ for some $N \geq 1$, where
$$\dom_{\Q} :=\{\lambda\in \mf[h]^*: \langle \lambda, \alpha_i^\vee\rangle \in \Q_+ \,\forall i\in I\}.$$ 
\end{definition}

For a finite-dimensional semisimple Lie algebra $\mathring{\mf[g]}$, the saturated tensor cone $\Gamma(\mathring{\mf[g]})$ has been studied extensively using geometric methods; we refer to the survey \cite{Ku2}. In this case, $\Gamma(\mathring{\mf[g]})$ is a closed polyhedral convex cone with inequalities determined by data coming from the cohomology of the associated flag varieties. An important notion in this topic is that of a saturation factor for $\Gamma(\mathring{\mf[g]})$, which we define below. 

\begin{definition} \label{satfactor}
We say that a positive integer $d \geq 1$ is a saturation factor of $\mathring{\mf[g]}$ (or equivalently of $\Gamma(\mathring{\mf[g]})$) if, for any $(\lambda, \mu; \nu) \in \Gamma(\mathring{\mf[g]})\cap (\dom)^3$, we have $V(d \nu) \subset V(d \lambda) \otimes V(d \mu)$. 
\end{definition}

For example, the saturation theorem of Knutson and Tao \cite{KT} says that $d=1$ is a saturation factor for $\Gamma(sl_{n+1})$. The saturation conjecture of Kapovich and Millson \cite{KM} posits that $d=1$ should also be a saturation factor for any  simply-laced simple Lie algebra; it is known that the minimal saturation factor for the non-simply laced types satisfy $d > 1$. 

When $\mf[g]$ is an affine Kac--Moody Lie algebra, $\Gamma(\mf[g])$ is no longer closed nor polyhedral (the semigroup is infinitely generated); however, a similar set of inequalities governing the points of $\Gamma(\mf[g])$ were conjectured by Brown and Kumar \cite{BK} and later proven by Ressayre \cite{Res} in this setting. Further, it is not a priori clear that a saturation factor $d$ (analogously defined as in Definition \ref{satfactor}) should exist for $\Gamma(\mf[g])$. Nevertheless, in loc. cit. Ressayre provides saturation factors in this setting, again building on the work of Brown and Kumar, who computed the saturation factors for $\mf[g]=A_1^{(1)}$. In the following  sections, we will rely only on the existence of a saturation factor for $\Gamma(\mf[g])$, and not make explicit use of their specific values apart from the case of $d=1$ for $\Gamma(sl_{n+1})$. 

Finally, we recall the definition of a $\delta$-maximal component $V(\nu) \subset V(\lambda) \otimes V(\mu)$, taken from \cite{BK}*{Definition 2.2} or \cite{JK}*{Definition 2.2}. 

\begin{definition} 
Let $V(\nu) \subset V(\lambda) \otimes V(\mu)$ be an irreducible component of the tensor product. We say that $V(\nu)$ is $\delta$-maximal if $V(\nu+k\delta) \not \subset V(\lambda) \otimes V(\mu)$ for any $k >0$. 
\end{definition}

\noindent The $\delta$-maximal components are natural tensor product analogues of the $\delta$-maximal  dominant weights $\dom_{max}(\lambda)$ of an irreducible representation $V(\lambda)$. Likewise, we can often recover  arbitrary components in $V(\lambda) \otimes V(\mu)$ by focusing on just the $\delta$-maximal components. Motivated by the similarities, we would like to understand the associated ``$\delta$-strings" in $\Gamma(\mf[g])$, which should correspond to subrepresentations $V(\nu-k\delta) \subset V(\lambda) \otimes V(\mu)$. Unfortunately, the exact analogue of Lemma \ref{unbroken} does not hold in general for the appearance of components $V(\nu-k\delta)$. However, it is not too far from being correct; the goal of the next section is to give the precise analogue of Lemma \ref{unbroken}.

\subsection{The Goddard--Kent--Olive construction} \label{GKO}
We next introduce the key technical tool which will allow us to study the components of $V(\rho) \otimes V(\rho)$, which is the Goddard--Kent--Olive (or GKO, for short) construction of the Virasoro algebra. The GKO construction,  which is particularly applicable in the study of branching problems and the tensor decomposition problem,  is a ``relative" version of the earlier Sugawara coset construction. For further details, we refer to \cite{KRR}*{Lecture 10}. First, we recall the definition of the Virasoro algebra, as follows. 

\begin{definition} \label{Vir}
The Virasoro algebra $Vir$ is a Lie algebra over $\C$ with basis $\{c, L_k : k \in \Z\}$ with commutator relations 
$$
[L_k, L_j]=(k-j)L_{k+j} + \frac{1}{12}(k^3-k)\delta_{k,-j}c, \ [Vir, c]=0.
$$
\end{definition}

Given a finite-dimensional semisimple Lie algebra $\mathring{\mf[g]}$, the Sugawara construction realizes $Vir$ as a subalgebra of $\widehat{U}(\mf[g])$, a particular completion of the universal enveloping algebra of the associated affine Kac--Moody Lie algebra. The utility of the construction is to examine highest weight irreducible representations $V(\lambda)$ of $\mf[g]$ as representations of the Virasoro algebra. When we consider this realization of the Virasoro algebra, we will denote the basis by $\{L_k^{\mf[g]}, c^{\mf[g]}\}$.

The GKO construction considers an embedding $\mathring{\mf[g]}_1 \hookrightarrow \mathring{\mf[g]}_2$ of finite-dimensional semisimple (or more generally, reductive) Lie algebras and the actions of their affinizations $\mf[g]_1$ and $\mf[g]_2$ on $\mf[g]_2$-modules. In particular, we can take the diagonal embedding $\mathring{\mf[g]} \hookrightarrow \mathring{\mf[g]} \oplus \mathring{\mf[g]}$. Then the Virasoro algebra realization on the $\mf[g] \oplus \mf[g]$-module $V(\lambda) \otimes V(\mu)$, for $\lambda, \mu \in \dom$ given by the GKO construction is
$$
\begin{aligned}
L^{GKO}_k&=L^{\mf[g]\oplus \mf[g]}_k-L^{\mf[g]}_k, \\
c^{GKO}&= c^{\mf[g]\oplus \mf[g]}-c^{\mf[g]}. 
\end{aligned}
$$
The following proposition, taken from \cite{KRR}, summarizes the key properties of this Virasoro action on a tensor product $V(\lambda) \otimes V(\mu)$.

\begin{proposition}[\cite{KRR}, Proposition 10.3] \label{GKOConstruct}
Let $\mf[g]$ be the affine Kac--Moody Lie algebra associated to a simple Lie algebra $\mathring{\mf[g]}$ and $\lambda$, $\mu \in \dom$ with levels $\langle \lambda,K\rangle=l$, $\langle \mu, K \rangle=m$. Then 
\begin{enumerate}
\item $V(\lambda) \otimes V(\mu)$ is a unitarizable Virasoro representation with nonnegative central charge 
$$
(\op{dim} \mathring{\mf[g]}) \left( \frac{l}{l+h^\vee}+\frac{m}{m+h^\vee}-\frac{l+m}{l+m+h^\vee} \right),
$$
where $h^\vee$ is the dual Coxeter number of $\mf[g]$ (\cite{Kac}*{Section 6.1}).

\item $L^{GKO}_0$ acts on $V(\lambda) \otimes V(\mu)$ by 
$$
\frac{1}{2} \left( \frac{(\lambda | \lambda+2\rho)}{l +h^\vee} + \frac{(\mu | \mu+2\rho)}{m+h^\vee}- \frac{\Omega}{l+m+h^\vee} \right),
$$
where $\Omega$ is the Casimir operator of $\mf[g]$ (\cite{Ku3}*{Section 1.5}) and $( \cdot | \cdot)$ is the normalized form on $\mf[h]^\ast$ as in \cite{Ku3}*{Lemma 13.1.8}. Recall (loc. cit.) that $\Omega$ acts on a component $V(\nu) \subset V(\lambda)\otimes V(\mu)$ via the scalar $(\nu|\nu+2\rho)$. 

\item For all $k$, $[L^{GKO}_k, \mf[g]']=0$; i.e., the $L^{GKO}_k$ are intertwining operators for the representation of $\mf[g]':=[\mf[g], \mf[g]]$ on $V(\lambda) \otimes V(\mu)$.
\end{enumerate}
\end{proposition}

In particular, given a $\delta$-maximal component $V(\nu) \subset V(\lambda) \otimes V(\mu)$, Proposition \ref{GKOConstruct}(3) says that $L^{GKO}_k$ stabilizes the $\mf[g]'$ isotypical component 
$$
\sum_{n \geq 0} V(\nu-n\delta)^{m_{\lambda, \mu}^{\nu-n\delta}} \subset V(\lambda) \otimes V(\mu),
$$
with $L^{GKO}_k\cdot v_\nu$ either zero or a highest weight vector for the $\mf[g]$-action, where $v_\nu$ is the highest weight vector of $V(\nu)$. We refer to Kac--Wakimoto \cite{KW} who considered the connection between the action of the GKO construction on these components and the nonvanishing of the multiplicities $m_{\lambda, \mu}^{\nu-k\delta}$. This was further used in the contexts of the affine tensor semigroup and affine tensor cone by Brown--Kumar \cite{BK} and Ressarye \cite{Res}. These methods were adapted and crucially used in our previous work \cite{JK} on the existence of root components in the tensor decomposition problem. The following proposition summarizes the technical application of this approach; we refer to \cite{JK}*{Section 3} for a more in-depth treatment and proofs.

\begin{proposition} \label{GKOMaximal}
Let $\lambda$, $\mu \in \dom$ with positive levels $l$, $m$, respectively. Let $V(\nu) \subset V(\lambda) \otimes V(\mu)$ be a $\delta$-maximal component. Then
\begin{enumerate}
\item If $ \frac{(\lambda | \lambda+2\rho)}{l +h^\vee} + \frac{(\mu | \mu+2\rho)}{m+h^\vee}- \frac{ (\nu | \nu+2\rho)}{l+m+h^\vee}  \neq 0$, then $V(\nu-k\delta) \subset V(\lambda) \otimes V(\mu)$ for all $k \geq 0$.
\item Else, $V(\nu-k\delta) \subset V(\lambda) \otimes V(\mu)$ for $k=0$, $k \geq 2$, and $V(\nu-\delta) \not \subset V(\lambda) \otimes V(\mu)$.
\end{enumerate}
\end{proposition}

Thus, the key conclusion is that by computing the action of $L^{GKO}_0$ on a  potential $\delta$-maximal component $V(\nu) \subset V(\lambda) \otimes V(\nu)$, we can show the existence of all submodules of the form $V(\nu-k\delta)$ by demonstrating the appearance of $V(\nu)$ in the tensor product. We note, however, that this does \textit{not} give a way to construct the $\delta$-maximal components themselves; this typically must be done independently of the Virasoro action considerations. These constructions, as well as verifying the positivity of the $L^{GKO}_0$ action, will form the basis of our approach when restricting to $\lambda=\mu=\rho$ in the following sections.

\section{Components of $V(d\rho) \otimes V(d\rho)$} 
In this section $\mf[g]$ is the affine Kac-Moody Lie algebra associated to a simple Lie algebra $\mathring{\mf[g]}$. 
\vskip1ex

With the previous details at our disposal, we can now consider the components of $V(d\rho) \otimes V(d\rho)$ for a saturation factor $d \geq 1$ of $\mf[g]$. In particular, recall that by definition (as in Definition \ref{satfactor}) if $\lambda \leq 2\rho$ is a dominant integral  weight such that $V(N\lambda) \subset V(N\rho) \otimes V(N\rho)$ for some $N \geq 1$, then for any saturation factor $d \geq 1$ we have $V(d\lambda) \subset V(d\rho) \otimes V(d\rho)$. Our approach is to first show, using the convexity of the saturated tensor semigroup, that such an $N$ exists for all dominant weights $\lambda \leq 2\rho$. First, we have the following proposition which is an adaptation of \cite{CKM}*{Proposition 9} to the affine setting.

\begin{proposition} \label{CKMprop} 
Let $\lambda \in \dom$ with $\lambda \leq 2\rho$. Then we can write $\lambda=\rho+\beta$ for some weight $\beta \in \mc[P](\rho)$.  
\end{proposition}

\begin{proof}
Since $\lambda \leq 2\rho$, we have that $\lambda-\rho \in \rho + \mc[Q]$. So, by Proposition \ref{convhull}, it suffices to show that $\lambda-\rho \in \op{conv}_\Q \{w \rho: w \in W\}$. Now, by Proposition \ref{dompoly} and Corollary \ref{2rhovert}, we can write 
$$
\lambda=\sum_{J \subsetneq I} c_J \hat{b}_J(2\rho) + c_\delta(-\delta),
$$
where as in Corollary \ref{2rhovert},
 $\hat{b}_J(2\rho)=\rho+w_0^J(\rho)$, and $c_J, c_\delta \in \Q_+$ with $\sum_{J} c_J=1$. Thus we have 
$$
\lambda-\rho=\sum_{J \subsetneq I} c_J w_0^J(\rho) +c_\delta(-\delta).
$$
Clearing denominators, we can write 
$$
N(\lambda-\rho)=\sum_{J \subsetneq I} c'_J w_0^J(\rho)+c'_\delta(-\delta),
$$
where $N \in \Z_{\geq 1}$, and $c'_J, c'_\delta \in \Z_+$ with $\sum_J c'_J=N$. Now, 
$$
\sum_{J \subsetneq I} c'_J w_0^J(\rho) =\sum_{J \subsetneq I} \frac{c'_J}{N} w_0^J(N\rho) 
\in \mc[P](N \rho),
$$
by Proposition \ref{convhull} since 
$$ N\rho- \sum_{J \subsetneq I} c'_J w_0^J(\rho)= -\sum_{J \subsetneq I} c'_J (w_0^J(\rho)-\rho)\in Q.$$
By Lemma \ref{unbroken}, we have that subtracting $c'_\delta (\delta)$ will remain a weight of $\mc[P](N\rho)$, so we conclude $N(\lambda-\rho) \in \mc[P](N\rho)$. Since $N(\lambda-\rho)$ is an integral weight in $\mc[P](N\rho)$, by Proposition \ref{convhull} we can write 
$$
N(\lambda-\rho)=\sum_{w \in W} a_w w(N\rho) = N\sum_{w \in W} a_w w(\rho)
$$
for some $a_w \in \Q_+$ with $\sum_{w \in W} a_w=1$. Dividing by $N$ and using Proposition  \ref{convhull} again
 gives the result since 
$$(\sum_{w\in W}\, a_w w(\rho)) - \rho=\lambda -2\rho \in Q\,\,\,\text{by assumption}.$$
\end{proof}

\begin{remark} \label{Mvertices}
By the same proof, since the candidates for vertices of the dominant weight polyhedron depend linearly on the highest weight, we have that for any integer $M \geq 1$, if $\lambda \in \dom$ is such that $\lambda \leq 2M\rho$, then we can write $\lambda=M\rho+\beta$ for some weight $\beta \in \mc[P](M\rho)$.
\end{remark}

\begin{remark}\label{remark5.3}
Using Proposition \ref{CKMprop}, one could alternatively adapt the proof of Theorem \ref{finite} as in \cite{CKM} in a straightforward way to prove our Theorem \ref{components}(1)  using the results of Ressayre \cite{Res} on the inequalities for the affine tensor semigroup. We do not take that approach, however, and instead highlight the use of the GKO construction, which will be key to the case $\mf[g]=A_n^{(1)}$.
\end{remark}

Now, we can use Proposition \ref{CKMprop} and Remark \ref{Mvertices} for $\lambda \leq 2M\rho$ to simplify the proof of the  following computational proposition. This, in turn, will allow us to apply the GKO construction and reduce the proof of Theorem \ref{components} to the $\delta$-maximal components by using  Proposition \ref{GKOMaximal}.

\begin{proposition} \label{L0weight}
Let $M \geq 1$ be an integer, and let $\lambda \in \dom$ with $\lambda \leq 2M\rho$. Then
$$
 \frac{2(M\rho| (M+2)\rho)}{(M+1)h^\vee}-\frac{(\lambda | \lambda + 2\rho)}{(2M+1)h^\vee}  >0.
$$
\end{proposition}

\begin{proof}
Write $\lambda=M\rho+\beta$ for some $\beta \in \mc[P](M\rho)$. Substituting this for $\lambda$ and by careful simplifications, we get 
$$
\begin{aligned}
 \frac{2(M\rho| (M+2)\rho)}{(M+1)h^\vee}-\frac{(\lambda | \lambda + 2\rho)}{(2M+1)h^\vee}  &=
\frac{1}{(M+1)(2M+1)h^\vee} \big( (2M+4) ( (M\rho | M \rho) - (M\rho | \beta) ) \\
 &+(M+1) ( (M \rho | M \rho) - (\beta | \beta) ) +2(M\rho | M\rho) +(2\rho | M\rho-\beta) \big). 
\end{aligned}
$$
But by \cite{Kac}*{Proposition 11.4}, we have that $(M\rho | M \rho)-(M \rho | \beta) \geq 0$ and $(M \rho | M \rho) - (\beta | \beta) \geq 0$, since $\beta \in \mc[P](M\rho)$. Finally, one can check that $2(M \rho | M \rho)=2M^2 (\rho | \rho) >0$, and $(2\rho | M\rho-\beta) \geq 0$, since $M\rho-\beta \in \mc[Q]^+$. Thus the expression of the proposition  is strictly positive. 
\end{proof}

With this in hand, we can apply Proposition \ref{GKOMaximal} to conclude the following. Observe that, by definition, $h^\vee := \langle \rho, K\rangle$. 

\begin{corollary} \label{2MrhoMaximals} Let $M$ be a positive integer. 
Take  $\lambda \in \dom$ such that $V(\lambda) \subset V(M\rho) \otimes V(M\rho)$ is a $\delta$-maximal component. Then $V(\lambda-k\delta) \subset V(M\rho) \otimes V(M\rho)$ for any $k \geq 0$.
\end{corollary}

We can now prove the following theorem, which is the first major result of this paper.

\begin{theorem} \label{maintheorem} 
Take $\lambda \in \dom$ such that $\lambda \leq 2\rho$. Then $V(d\lambda) \subset V(d\rho) \otimes V(d\rho)$ for any saturation factor $d$ of $\mf[g]$. 
\end{theorem}

\begin{proof}
As in the proof of   Proposition \ref{CKMprop}, write 
$$
\lambda=\sum_{J \subsetneq I} c_J \hat{b}_J(2\rho) + c_\delta(-\delta),\,\,\,\text{for $c_J, c_\delta\in \Q_+$}.
$$
Again clearing denominators, we write (for some integer $N\geq 1$)
$$
N \lambda = \sum_{J \subsetneq I} c'_J \hat{b}_J(2\rho) + c'_\delta(-\delta)
$$
with each $c'_J, c'_\delta \in \Z_+$ and $\sum_J c'_J=N$. But since $c'_J \hat{b}_J(2\rho)=c'_J\rho+w_0^J(c'_J\rho)$ (cf. Corollary \ref{2rhovert}), via the PRV components of Theorem \ref{PRV} we get that $V(c'_J\hat{b}_J(2\rho)) \subset V(c'_J \rho) \otimes V(c'_J \rho)$ for all $J$. (Observe that $\rho+w_0^J(\rho)$ is a dominant weight.)   Thus by additivity in the tensor semigroup, we get 
$$
V(\Lambda) \subset V(N\rho) \otimes V(N\rho),
$$
where $\Lambda:=\sum_{J \subsetneq I} c'_J \hat{b}_J(2\rho)$. Finally, by  Corollary \ref{2MrhoMaximals}, we get (by replacing $\Lambda$ with the corresponding $\delta$-maximal component) that 
$$
V(N\lambda)=V(\Lambda-c'_\delta(\delta)) \subset V(N \rho) \otimes V(N\rho).
$$
Thus by the existence of a saturation factor for the affine tensor cone, we have that $V(d\lambda) \subset V(d\rho) \otimes V(d\rho)$ for any saturation factor $d$. 
\end{proof}

\begin{remark} Assume that  $\lambda \in \dom$,  $\lambda \leq 2\rho$ is such that 
 $$V(d\lambda) \subset V(d\rho) \otimes V(d\rho),\,\,\,\text{ for some  $d\in \Z_{>0}$}. $$
Then, for any $k\in \Z$ and  $\lambda' \in \dom$ such that   $\lambda' \leq 2(\rho+k\delta)$,
 $$V(d\lambda') \subset V(d(\rho+k\delta)) \otimes V(d(\rho+k\delta)).$$
To prove this, observe that for any $ \mu \in \dom$ and $k\in \Z$,
$$V(\mu+k\delta) \simeq V(\mu)\otimes V(k\delta).$$
\end{remark}

\section{The case $\mf[g]=A_n^{(1)}$ and conjecture for other types}

In this section, we restrict our attention to the case when $\mf[g]=A_n^{(1)}=\widehat{sl}_{n+1}$. By Theorem \ref{maintheorem}, we know that for any saturation factor $d$ for $\mf[g]$ we have that $V(d\lambda) \subset V(d\rho) \otimes V(d\rho)$ for all dominant $\lambda \leq 2\rho$. However, unlike in the corresponding finite-dimensional simple Lie algebra $\mathring{\mf[g]}=sl_{n+1}$, the saturation factor $d$ \textit{cannot} be taken to be $d=1$. It is known (\cite{Res}*{Theorem 3}) that in type $A_n^{(1)}$, any $d \geq 2$ is a saturation factor.

While the method of proof of Theorem \ref{maintheorem} applied to a finite-dimensional semisimple Lie algebra gives a new proof of Theorem \ref{finite}, we will assume the validity of  Theorem \ref{finite}  and use it to prove the corresponding result for $A_n^{(1)}$. By Corollary \ref{2MrhoMaximals}, in order to prove the following Theorem \ref{maintheorem2} it suffices to consider the case when $\lambda \in \mc[P]^+_{max}(2\rho)$. Since for $A_n^{(1)}$ we have that $\delta=\alpha_0+\alpha_1+\cdots+\alpha_n$, by Proposition \ref{maxsupport}, we know that any $\delta$-maximal dominant weight $\lambda \leq 2\rho$ satisfies 
$$
\op{supp}(2\rho-\lambda) \subsetneq I =\{0, 1, \dots, n\},
$$
where, for any $\gamma \in \mc[Q]^+$, we define the support of $\gamma$ as $\op{supp}(\gamma) = \{ i \in I: c_i \neq 0 \ \text{where } \gamma=\sum_{i \in I} c_i \alpha_i\}$.
We first give the following lemma, for certain such $\lambda$, following a similar approach as in \cite{JK}*{Proof of Proposition 6.4}.

\begin{lemma} \label{finitesupport}

Let $\lambda \in \mc[P]^+_{max}(2\rho)$ such that $2\rho-\lambda \in \mathring{\mc[Q]}^+$. Then $V(\lambda) \subset V(\rho) \otimes V(\rho)$. 
\end{lemma}

\begin{proof}
Let $\mathring{V}(\rho)$ be the $\mathring{\mf[g]}$-submodule of $V(\rho)$ generated by the highest weight vector $v_\rho \in V(\rho)$, and similarly for $\mathring{V}(\lambda)$. Then by Theorem \ref{finite}, we know that 
$$
\mathring{V}(\lambda) \subset \mathring{V}(\rho) \otimes \mathring{V}(\rho).
$$
Let $v_\lambda \in V(\rho) \otimes V(\rho)$ be a corresponding highest weight vector for $\mathring{\mf[g]}$ that generates $\mathring{V}(\lambda)$. We claim that this is in fact a highest weight vector for $\mf[g]$. All that remains to check is that $e_0.v_\lambda=0$, where $e_0$ is a root vector corresponding to the positive root $\alpha_0$. But by convention, we have that $\rho(d)=0$, so that no vector in $V(\rho) \otimes V(\rho)$ can have weight $\nu$ with $\nu(d) \geq 1$. And, since $2\rho-\lambda \in \mathring{\mc[Q]}^+$, we know that $\lambda(d)=0$. Then if $e_0.v_\lambda \neq 0$, we have that $e_0.v_\lambda$ has weight satisfying $(\lambda+\alpha_0)(d)=\alpha_0(d)=1$, a contradiction. Thus $e_0.v_\lambda=0$, so that $V(\lambda) \subset V(\rho) \otimes V(\rho)$. 
\end{proof}

By the symmetry of the Dynkin diagram of type $A_n^{(1)}$, Lemma \ref{finitesupport} actually suffices to handle the case of all $\delta$-maximal weights $\lambda \in \mc[P]_{max}^+(2\rho)$. Together with Corollary \ref{2MrhoMaximals} for $M=1$, we are able to prove the following theorem.

\begin{theorem} \label{maintheorem2}
Let $\mf[g]=A_n^{(1)}$, and let $\lambda \in \dom$ with $\lambda \leq 2\rho$. Then $V(\lambda) \subset V(\rho) \otimes V(\rho)$. 
\end{theorem}

\begin{proof}
First, let $\lambda \in \mc[P]^+_{\max}(2\rho)$ be a $\delta$-maximal weight. We know that $\op{supp}(2\rho-\lambda) \subsetneq I$. Consider any diagram automorphism $\sigma$ of the Dynkin diagram of $A_n^{(1)}$ such that $\op{supp}(\sigma(2\rho-\lambda)) \subset I \backslash \{0\}$. Then $\sigma$ induces an automorphism of $\mf[g]$ and also of $V(\rho)$ and $V(\rho) \otimes V(\rho)$ (since $\rho$ is stable under the diagram automorphism). If $\lambda':=\sigma(\lambda) \in \dom$ is the corresponding  dominant integral  weight, we thus have $\lambda' \leq 2\rho$ and $\op{supp}(2\rho-\lambda') \subset I \backslash \{0\}$, so by  Lemma \ref{finitesupport} we get that $V(\lambda') \subset V(\rho) \otimes V(\rho)$. Applying $\sigma^{-1}$, we get that $V(\lambda) \subset V(\rho) \otimes V(\rho)$. Now, the theorem follows by  applying  Corollary \ref{2MrhoMaximals}.
\end{proof}

For $\mf[g]$ of other types, we do not have sufficient diagram automorphisms to reduce to those $\lambda$ with $2\rho -\lambda \in \mathring{\mc[Q]}^+$. Even if this were the case, as Conjecture \ref{Kostant} is only known to hold  in types $A_n$ and the exceptional types, the proof method used for Theorem \ref{maintheorem2} to leverage the result for $\mathring{\mf[g]}$ is not easily adaptable.

However, direct computer-verified computations via the software Sage \cite{TSD}  show that Theorem \ref{maintheorem2} extends for low rank examples like $\mf[g]=B^{(1)}_2$, $G^{(1)}_2$, $B_3^{(1)}$, and $D^{(1)}_4$. Thus, we conclude by proposing the following analogue of Kostant's conjecture for affine Kac--Moody Lie algebras.

\begin{conjecture} \label{affineconjecture} 
Let $\mf[g]$ be any untwisted affine Kac--Moody Lie algebra, and $\lambda \in \dom$ such that $\lambda \leq 2\rho$. Then $V(\lambda) \subset V(\rho) \otimes V(\rho)$.
\end{conjecture}

\section*{Conflict of Interest}
The authors declare that they have no conflict of interest.



\begin{bibdiv}
\begin{biblist}








\bib{BS}{article}{
	AUTHOR={Berenstein, A.}
	AUTHOR={Sjamaar, R.}
	TITLE={Coadjoint orbits, moment polytopes, and the Hilbert--Mumford criterion}
	JOURNAL={Journal of the American Mathematical Society}
	VOLUME={13}
	YEAR={2000}
}


\bib{BZ}{article}{
	AUTHOR={Berenstein, A.}
	AUTHOR={Zelevinsky, A.}
	TITLE={Triple multiplicities for sl(r+1) and the spectrum of the exterior algebra of the adjoint representation}
	JOURNAL={J. Algebraic Combinatorics}
	VOLUME={1}
	YEAR={1992}
}


\bib{BJK}{article}{
	AUTHOR={Besson, M.}
	AUTHOR={Jeralds, S.}
	AUTHOR={Kiers, J.}
	TITLE={Vertices of intersection polytopes and rays of generalized Kostka cones}
	JOURNAL={Journal of Lie Theory}
	VOLUME={31}
	ISSUE={4}
	YEAR={2021}
}


\bib{BK}{article}{
	AUTHOR={Brown, M.}
	AUTHOR={Kumar, S.}
	TITLE={A study of saturated tensor cone for symmetrizable Kac--Moody algebras}
	JOURNAL={Mathematische Annalen}
	VOLUME={360}
	YEAR={2014}
}


\bib{CKM}{article}{
	AUTHOR={Chirivi, R.}
	AUTHOR={Kumar, S.}
	AUTHOR={Maffei, A.}
	TITLE={Components of $V(\rho) \otimes V(\rho)$}
	JOURNAL={Transformation Groups}
	VOLUME={22}
	YEAR={2017}
}






\bib{JK}{article}{
	AUTHOR={Jeralds, S.}
	AUTHOR={Kumar, S.}
	TITLE={Root components for tensor product of affine Kac--Moody Lie algebra modules}
	JOURNAL={Representation Theory (An electronic journal of the American Mathematical Society)}
	VOLUME={26}
	YEAR={2022}
}




\bib{Kac}{book}{
	AUTHOR={Kac, V.}
	TITLE={Infinite dimensional Lie algebras}
	PUBLISHER={Cambridge University Press}
	YEAR={1990}
}


\bib{KRR}{book}{
	AUTHOR={Kac, V.}
	AUTHOR={Raina, A.}
	AUTHOR={Rozhkovskaya}
	TITLE={Bombay lectures on highest weight representations of infinite dimensional Lie algebras}
	SERIES={Advanced Series in Mathematical Physics}
	VOLUME={29}
	PUBLISHER={World Scientific Publishing Co. Pte. Ltd., Hackensack, NJ}
	YEAR={2013}
}


\bib{KW}{article}{
	AUTHOR={Kac, V.}
	AUTHOR={Wakimoto, M.}
	TITLE={Modular and conformal invariance constraints in representation theory of affine algebras}
	JOURNAL={Advances in Mathematics}
	VOLUME={70}
	YEAR={1988}
}






\bib{KM}{article}{
	AUTHOR={Kapovich, M.}
	AUTHOR={Millson, J. J.}
	TITLE={Structure of the tensor product semigroup}
	JOURNAL={Asian Journal of Mathematics}
	VOLUME={10}
	YEAR={2006}
}





\bib{Kos}{article}{
	AUTHOR={Kostant, B.}
	TITLE={Clifford algebra analogue of the Hopf--Koszul--Samelson theorem, the $\rho$-decomposition of $C(\mathfrak{g})= \End(V_\rho) \otimes C(P)$, and the $\mathfrak{g}$-module structure of $\bigwedge \mathfrak{g}$}
	JOURNAL={Adv. Math.}
	VOLUME={125}
	YEAR={1997}
}



\bib{KT}{article}{
	AUTHOR={Knutson, A.}
	AUTHOR={Tao, T.}
	TITLE={The honeycomb model of $GL_n(\mathbb{C})$ tensor products I: Proof of the saturation conjecture}
	JOURNAL={Journal of the American Mathematical Society}
	VOLUME={12}
	YEAR={1999}
}

\bib{Ku0}{article}{
	AUTHOR={Kumar, S.}
	TITLE={Proof of the Parthasarathy-RangaRao-Varadarajan conjecture}
	JOURNAL={Invent. Math.}
	VOLUME={93}
	YEAR={1988}
}

\bib{Ku1}{article}{
	AUTHOR={Kumar, S.}
	TITLE={Existence of certain components in the tensor product of two integrable highest weight modules for Kac--Moody algebras}
	JOURNAL={Infinite-dimensional Lie algebras and groups (Luminy-Marseille, 1988), Adv. Ser. Math. Phys.}
	VOLUME={7} 
	PUBLISHER={World Sci. Publ., Teaneck, NJ, 1989}
	YEAR={1989}
}




\bib{Ku3}{book}{
	AUTHOR={Kumar, S.}
	TITLE={Kac--Moody groups, their flag varieties, and representation theory}
	PUBLISHER={Birkh\"auser}
	SERIES={Progress in Mathematics}
	VOLUME={204}
	YEAR={2002}
}




\bib{Ku2}{article}{
	AUTHOR={Kumar, S.}
	TITLE={A survey of the additive eigenvalue problem (with appendix by M. Kapovich)}
	JOURNAL={Transformation Groups}
	VOLUME={19}
	YEAR={2014}
}










\bib{Mat}{article}{
	AUTHOR={Mathieu, O.}
	TITLE={Construction d'un grouupe de Kac--Moody et applications}
	JOURNAL={Compositio Math.}
	VOLUME={69}
	YEAR={1989}
}




\bib{NP}{article}{
	AUTHOR={Nadimpalli, S.}
	AUTHOR={Pattanayak, S.}
	TITLE={A note on branching of $V(\rho)$}
	JOURNAL={Journal of Algebra}
	VOLUME={594}
	YEAR={2022}
}




\bib{Res}{article}{
	AUTHOR={Ressayre, N.}
	TITLE={On the tensor semigroup of affine Kac--Moody Lie algebras}
	JOURNAL={J. Amer. Math. Soc.}
	VOLUME={35}
	YEAR={2022}
}


\bib{Sch}{book}{
	AUTHOR={Schrijver, A.}
	TITLE={Theory of linear and integer programming}
	SERIES={Wiley--Interscience Series in Discrete Mathematics and Optimization}
	PUBLISHER={John Wiley \& Sons}
	YEAR={1999}
}


\bib{TSD}{article}{
	AUTHOR={The Sage Developers}
	TITLE={SageMath, the Sage Mathematics Software System (Version 8.0)}
	YEAR={2017}
	NOTE={\url{http://www.sagemath.org}}
}





\end{biblist}
\end{bibdiv}
\end{document}